\documentclass[11pt, reqno]{amsart}
\usepackage{amsfonts}
\usepackage{amsmath}
\usepackage{amssymb, color}
\usepackage{amscd}
\usepackage[mathscr]{eucal}
\usepackage{url}

\allowdisplaybreaks[1]

\theoremstyle{plain}

\newtheorem{theorem}{Theorem}[section]

\newtheorem{cor}[theorem]{Corollary}
\theoremstyle{definition}

\numberwithin{equation}{section}

\makeatletter
\def\imod#1{\allowbreak\mkern5mu({\operator@font mod}\,\,#1)}
\makeatother

\begin{document}

\title[Ap{\'e}ry-like numbers and families of CM newforms]
{Ap{\'e}ry-like numbers and families of\\ newforms with complex multiplication}

\author{Alexis Gomez, Dermot McCarthy, Dylan Young}

\address{Department of Mathematics $\&$ Statistics, Texas Tech University, Lubbock, TX 79410-1042, USA}

\email{Alexis.Gomez@ttu.edu}

\email{dermot.mccarthy@ttu.edu}

\email{Dylan.Young@ttu.edu}


\begin{abstract}
Using Hecke characters, we construct two infinite families of newforms with complex multiplication, one by $\mathbb{Q}(\sqrt{-3})$ and the other by $\mathbb{Q}(\sqrt{-2})$. The values of the $p$-th Fourier coefficients of all the forms in each family can be described by a single formula, which we provide explicitly.  This allows us to establish a formula relating the $p$-th Fourier coefficients of forms of different weights, within each family. We then prove congruence relations between the $p$-th Fourier coefficients of these newforms at all odd weights and values coming from two of Zagier's sporadic Ap{\'e}ry-like sequences.
\end{abstract}

\maketitle

%%%%%%%%%%%%%%%%%%%%%%%%%%%%%%%%%%%%%%%%%%%%%%%%%%
%%%%%%%%%%%%%%%%%%%%%%%%%%%%%%%%%%%%%%%%%%%%%%%%%%
%%%%%%%%%%%%%%%%%%%%%%%%%%%%%%%%%%%%%%%%%%%%%%%%%%
%%%%%%%%%%%%%%%%%%%%%%%%%%%%%%%%%%%%%%%%%%%%%%%%%%
%%%%%%%%%%%%%%%%%%%%%%%%%%%%%%%%%%%%%%%%%%%%%%%%%%
%%%%%%%%%%%%%%%%%%%%%%%%%%%%%%%%%%%%%%%%%%%%%%%%%%

\section{Introduction and Statement of Results}\label{sec_Intro}

In 1978, Ap{\'e}ry \cite{Ap} provided quite an unexpected proof of the irrationality of $\zeta(3)$. His methods also yielded another proof of the irrationality of $\zeta(2)$. The sequences 
\begin{equation}\label{AperyNumbers}
  A(n) = \sum_{k=0}^{n} \binom{n}{k}^2 \binom{n+k}{k}, \quad \quad \quad \quad B(n) = \sum_{k=0}^{n} \binom{n}{k}^2 \binom{n+k}{k}^2
\end{equation}
arose in those proofs and are now commonly referred to as the Ap{\'e}ry numbers. Since then, there has been substantial interest in both the intrinsic arithmetic properties of the Ap{\'e}ry numbers and their relationship to modular forms. For example, consider the unique newform in $S_{3}(\Gamma_{0}(16),(\tfrac{-4}{\cdot}))$, which has complex multiplication (CM) by $\mathbb{Q}(\sqrt{-1})$,
\begin{equation*}
h(z) = q \prod_{n\ge1} (1 - q^{4n})^6 =: \sum_{n=1}^{\infty} \alpha(n) q^n,
\end{equation*}
where
$q:=e^{2\pi i z}$ and $z \in \mathbb{H}$.
Ahlgren \cite{A} proved that, for all primes $p \geq 5$, 
\begin{equation} \label{superapery1}
  A\Bigl(\frac{p-1}{2}\Bigr) \equiv \alpha(p) \pmod{p^2},
\end{equation}
thus confirming a conjecture in \cite{SB}. A similar result applies to the $B(n)$ sequence and a newform in $S_{4}(\Gamma_{0}(8))$ \cite{AO}.

Inspired by Ap{\'e}ry, Beukers \cite{Be} gave another proof of the irrationality of $\zeta(2)$ and $\zeta(3)$ by considering certain families of integrals which satisfy $\mathbb{Q}$-linear relations with the respective zeta values. The  Ap{\'e}ry numbers appear as coefficients in these relations. More recently, Brown \cite{B} introduced {\it cellular integrals} which can be expressed as $\mathbb{Q}$-linear combinations of multiple zeta values, and include Beukers integrals as a special cases.
In \cite{MOS}, the authors, of whom the second author of this paper is one, examine sequences arising from the coefficients in these linear combinations. They show that all powers of the Ap{\'e}ry numbers are among these sequences and that they too satisfy congruence relations with Fourier coefficients of modular forms which have complex multiplication by $\mathbb{Q}(\sqrt{-1})$, similar to (\ref{superapery1}). To do this, they construct, using Hecke characters, an infinite family of modular forms, which have complex multiplication by $\mathbb{Q}(\sqrt{-1})$, and whose $p$-th Fourier coefficients can be described by a single formula. Specifically, for every positive integer $k \geq 2$, there exists a weight $k$ CM newform $h_k(z) =: \sum_{n=1}^{\infty} \alpha_k(n) q^n \in S_k(\Gamma_0(N_k), \chi_k)$ such that
$$\alpha_k(p)
=
\begin{cases}
(-1)^{\frac{(x+y-1)(k-1)}{2}} \left[(x+iy)^{k-1} + (x-iy)^{k-1} \right], & \text{if $p \equiv 1 \imod 4$, $p= x^2+y^2$, $x$ odd},\\
0, & \text{if $p \equiv 3 \imod 4$}.
\end{cases}
$$
The level $N_k$ and character $\chi_k$ are explicitly stated and depend only the congruence class of $k$ modulo 4.
Using this formula for the $p$-th Fourier coefficients, nice congruence relations between Fourier coefficients of forms of different weight can be found. For example,
\begin{equation}\label{for_alphap2}
 \alpha_3(p)^s 
\equiv \alpha_{t}(p) \pmod{p^2},
\end{equation}
where $t=2s+1$. Then, using (\ref{superapery1}), we get that 
\begin{equation}\label{superapery2}
 A\Bigl(\frac{p-1}{2}\Bigr)^s 
\equiv \alpha_{t}(p) \pmod{p^2}.
\end{equation}

In this paper, we perform similar exercises with, firstly, a family of modular forms with complex multiplication by $\mathbb{Q}(\sqrt{-3})$ and the sequence 
\begin{equation}\label{def_Cn}
C(n) = \sum_{k=0}^{n} \binom{n}{k}^2 \binom{2k}{k},
\end{equation}
and then with a family of modular forms with complex multiplication by $\mathbb{Q}(\sqrt{-2})$ and the sequence 
\begin{equation}\label{def_Dn}
D(n) = \sum_{k=0}^{n} \binom{n}{k}^3.
\end{equation}

The sequences $C(n)$ and $D(n)$ are two of the so-called sporadic Ap{\'e}ry-like numbers found by Zagier \cite{Z}.  We will discuss the Ap{\'e}ry-like numbers and some properties of $C(n)$ and $D(n)$ in Section \ref{sec_Prelim_Apery}.

The results of this paper are as follows. First, using Hecke characters, we construct an infinite family of newforms which have complex multiplication by $\mathbb{Q}(\sqrt{-3})$ and whose $p$-th Fourier coefficients can be described by a single formula.  This is slightly more complicated than the $\mathbb{Q}(\sqrt{-1})$ case, as the ring of integers of $\mathbb{Q}(\sqrt{-3})$ admits half integers and its group of units is larger. Also, we need to consider a larger number of congruence classes for $k$, each of which requires its own construction.
\begin{theorem}\label{thm_CM}
Let $k \geq 2$ be an integer. Then there exists a weight $k$ newform with CM by $\mathbb{Q}(\sqrt{-3})$
$$
f_k(z) =: \sum_{n=1}^{\infty} \gamma_k(n) q^n
\in
\begin{cases}
S_k(\Gamma_0(36)), & \text{if $k \equiv 0,2 \imod 6$},\\
S_k(\Gamma_0(3),(\frac{-3}{\cdot})), & \text{if $k \equiv 1 \imod 6$},\\
S_k(\Gamma_0(12),(\frac{-3}{\cdot})), & \text{if $k \equiv 3,5 \imod 6$},\\
S_k(\Gamma_0(9)), & \text{if $k \equiv 4 \imod 6$}.
\end{cases}
$$
Furthermore, for any odd prime $p$,
$$
\gamma_k(p)
=
\begin{cases}
(a+ b\sqrt{-3})^{k-1} + (a-b \sqrt{-3})^{k-1} & \text{if $p \equiv 1 \imod 6$},\\
(-3)^\frac{k-1}{2} & \text{if $p =3$ and $k$ odd}\\
0 & \text{otherwise},
\end{cases}
$$
where $p=a^2+3b^2$ when $p \equiv 1 \pmod{6}$, with $a\equiv 1 \imod 3$ and $b>0$ integers.
Therefore, for all $n\geq 1, k \geq 2$, $\gamma_k(n)$ is an integer.
\end{theorem}

We then get some nice relations between the Fourier coefficients of forms of different weights. 
\begin{cor}\label{cor_Coeffs}
Let $k \geq 2$ be an integer and let
$f_k(z) =: \sum_{n=1}^{\infty} \gamma_k(n) q^n$
be the weight $k$ CM newform described in Theorem~\ref{thm_CM}. 
Then, for any prime $p>3$ and integer $r \geq 1$,
\begin{align*}
\gamma_k(p)^r &= \sum_{t=0}^{\lfloor \frac{r-1}{2} \rfloor} \binom{r}{t} p^{t(k-1)} \gamma_{(r-2t)(k-1)+1}(p)
+ 
\begin{cases}
\displaystyle \binom{r}{\frac{r}{2}} p^{\frac{r}{2}(k-1)}, & \text{if $p \equiv 1 \imod 6$ and $r$ even},\\
0, & \text{otherwise}.
\end{cases}
\end{align*}
Also, for $k$ odd,
$$\gamma_k(3)^r = \gamma_{r(k-1)+1}(3).$$
\end{cor}

\begin{cor}\label{cor_CongCoeffs}
Let $k \geq 2$ be an integer and let
$f_k(z) =: \sum_{n=1}^{\infty} \gamma_k(n) q^n$
be the weight $k$ CM newform described in Theorem~\ref{thm_CM}. 
Then, for any prime $p>3$ and integer $r \geq 1$,
\begin{equation*}
\gamma_k(p)^r \equiv \gamma_{r(k-1)+1}(p) \pmod{p^{k-1}}.
\end{equation*}
\end{cor}

Results in \cite{SB} provide us with a congruence relation between values of $C(n)$ and the Fourier coefficients of the weight three form, $f_3(z)$, from Theorem \ref{thm_CM}. Combining this with Corollary \ref{cor_CongCoeffs} we  relate values of $C(n)$ to all $f_k(z)$ for $k$ odd.

\begin{theorem}\label{thm_CongCn}
Let $k \geq 2$ be an integer and let
$f_k(z) =: \sum_{n=1}^{\infty} \gamma_k(n) q^n$
be the weight $k$ CM newform described in Theorem~\ref{thm_CM}. 
Let $C(n)$ be the sequence defined in (\ref{def_Cn}).
Then, for any prime $p>3$ and integer $r \geq 1$,
\begin{equation*}
C\left(\frac{p^r-1}{2}\right)
\equiv \gamma_{2r+1}(p) \pmod{p}.
\end{equation*}
\end{theorem}
\noindent We note that the congruence relation in Theorem \ref{thm_CongCn} does not hold modulo $p^2$.

Next we construct an infinite family of newforms which have complex multiplication by $\mathbb{Q}(\sqrt{-2})$ and establish similar results with respect to the $D(n)$ sequence.
\begin{theorem}\label{thm_CM2}
Let $k \geq 3$ be an \textbf{odd} integer. Then there exists a weight $k$ newform with CM by $\mathbb{Q}(\sqrt{-2})$,
$$
g_k(z) =: \sum_{n=1}^{\infty} \beta_k(n) q^n
\in
S_k(\Gamma_0(8),(\tfrac{-8}{\cdot}))
$$
such that for any prime $p$,
$$
\beta_k(p)
=
\begin{cases}
(c+d \sqrt{-2})^{k-1} + (c-d \sqrt{-2})^{k-1} & \text{if $p \equiv 1,3 \imod 8$},\\
(-2)^\frac{k-1}{2} & \text{if $p =2$}\\
0 & \text{otherwise},
\end{cases}
$$
where $p=c^2+2d^2$ when $p \equiv 1,3 \pmod{8}$, with $c$ and $d$ integers.
Therefore, for all $n\geq 1, k \geq 3$, $\beta_k(n)$ is an integer.
\end{theorem}

\begin{cor}\label{cor_Coeffs2}
Let $k \geq 3$ be an odd integer and let
$g_k(z) =: \sum_{n=1}^{\infty} \beta_k(n) q^n$
be the weight $k$ CM newform described in Theorem~\ref{thm_CM2}. 
Then, for any prime $p>2$ and integer $r \geq 1$,
\begin{align*}
\beta_k(p)^r &= \sum_{t=0}^{\lfloor \frac{r-1}{2} \rfloor} \binom{r}{t} p^{t(k-1)} \beta_{(r-2t)(k-1)+1}(p)
+ 
\begin{cases}
\displaystyle \binom{r}{\frac{r}{2}} p^{\frac{r}{2}(k-1)}, & \text{if $p \equiv 1,3  \imod 8$ and $r$ even},\\
0, & \text{otherwise}.
\end{cases}
\end{align*}
Also,
$$\beta_k(2)^r = \beta_{r(k-1)+1}(2).$$
\end{cor}

\begin{cor}\label{cor_CongCoeffs2}
Let $k \geq 3$ be an odd integer and let
$g_k(z) =: \sum_{n=1}^{\infty} \beta_k(n) q^n$
be the weight $k$ CM newform described in Theorem~\ref{thm_CM2}. 
Then, for any prime $p>2$ and integer $r \geq 1$,
\begin{equation*}
\beta_k(p)^r \equiv \beta_{r(k-1)+1}(p) \pmod{p^{k-1}}.
\end{equation*}
\end{cor}

\begin{theorem}\label{thm_CongCn2}
Let $k \geq 3$ be an odd integer and let
$g_k(z) =: \sum_{n=1}^{\infty} \beta_k(n) q^n$
be the weight $k$ CM newform described in Theorem~\ref{thm_CM2}. 
Let $D(n)$ be the sequence defined in (\ref{def_Dn}).
Then, for any prime $p>2$ and integer $r \geq 1$,
\begin{equation*}
(-1)^{\frac{p^r-1}{2}} D\left(\frac{p^r-1}{2}\right) 
\equiv\beta_{2r+1}(p) \pmod{p}.
\end{equation*}
\end{theorem}
\noindent We note that the congruence relation in Theorem \ref{thm_CongCn2} does not hold modulo $p^2$.

Applying the methods used to prove Theorems \ref{thm_CongCn} and \ref{thm_CongCn2}, to the results in \cite{MOS}, gives a slight enhancement of (\ref{superapery1}).
\begin{theorem}\label{thm_CongCn3}
Let $k \geq 2$ be an integer and let
$h_k(z) =: \sum_{n=1}^{\infty} \alpha_k(n) q^n$
be the weight $k$ CM newform described in Section \ref{sec_Intro}. 
Let $A(n)$ be the sequence defined in (\ref{AperyNumbers}).
Then, for any prime $p\geq 5$ and integer $r \geq 1$,
\begin{equation}
 A\Bigl(\frac{p^r-1}{2}\Bigr) 
\equiv \gamma_{2r+1}(p) \pmod{p^2}.
\end{equation}
\end{theorem}
 
%%%%%%%%%%%%%%%%%%%%%%%%%%%%%%%%%%%%%%%%%%%%%%%%%%
%%%%%%%%%%%%%%%%%%%%%%%%%%%%%%%%%%%%%%%%%%%%%%%%%%
%%%%%%%%%%%%%%%%%%%%%%%%%%%%%%%%%%%%%%%%%%%%%%%%%%
%%%%%%%%%%%%%%%%%%%%%%%%%%%%%%%%%%%%%%%%%%%%%%%%%%
%%%%%%%%%%%%%%%%%%%%%%%%%%%%%%%%%%%%%%%%%%%%%%%%%%
%%%%%%%%%%%%%%%%%%%%%%%%%%%%%%%%%%%%%%%%%%%%%%%%%%

\section{Preliminaries}\label{sec_Prelim}

\subsection{Modular forms with complex multiplication and Hecke characters}\label{sec_Prelim_CM}
In this section, we recall some properties of modular forms with complex multiplication and Hecke characters. For more details, see \cite{Ri}.

Suppose $\psi$ is a nontrivial real Dirichlet character with corresponding quadratic field $K$. A newform $f(z)= \sum_{n=1}^{\infty} \gamma(n) \, q^n$, where $q := e^{2 \pi i z}$, has complex multiplication (CM) by $\psi$, or by $K$, if $\gamma(p) = \psi(p) \, \gamma(p)$ for all primes $p$ in a set of density one. 

By the work of Hecke and Shimura we can construct CM newforms using Hecke characters. Let $K=\mathbb{Q}(\sqrt{-d})$ be an imaginary quadratic field with discriminant $D$, and let $\mathcal{O}_K$ be its ring of integers. For an ideal $\mathfrak{f} \in \mathcal{O}_K$, let $I(\mathfrak{f})$ denote the group of fractional ideals prime to $\mathfrak{f}$. A Hecke character of weight $k$ and modulo $\mathfrak{f}$ is a homomorphism
$\displaystyle
\Phi : I(\mathfrak{f}) \to \mathbb{C}^{*}$,
satisfying
$\displaystyle
\Phi(\alpha \, \mathcal{O}_K)= \alpha^{k-1} 
$
when
$
\alpha \equiv^{\times} 1 \pmod {\mathfrak{f}}.
$
Let $N(\mathfrak{a})$ denote the norm of the ideal $\mathfrak{a}$. Then,
\begin{align*}
  f(z) := \sum_{\mathfrak{a}} \Phi(\mathfrak{a}) q^{N(\mathfrak{a})} = \sum_{n=1}^{\infty} \gamma(n) q^n,
\end{align*}
where the sum is over all ideals $\mathfrak{a}$ in $\mathcal{O}_K$ prime to $\mathfrak{f}$, is a Hecke eigenform of weight $k$ on ${\Gamma_0(|D| \cdot N(\mathfrak{f}))}$ with Nebentypus 
$\chi(n)=\left(\frac{D}{n}\right) \frac{\Phi(n \, \mathcal{O}_K)}{n^{k-1}}$. Here, $\left( \tfrac{a}{n}\right)$ is the Kronecker symbol.
Furthermore, $f$ has CM by $K$. We call $\mathfrak{f}$ the conductor of $\Phi$ if $\mathfrak{f}$ is minimal, i.e., if $\Phi$ is defined modulo $\mathfrak{f}^{\prime}$ then $\mathfrak{f} \mid \mathfrak{f}^{\prime}$. If $\mathfrak{f}$ is the conductor of $\Phi$ then $f(z)$ is a newform. From \cite{Ri}, we also know that every CM newform comes from a Hecke character in this way.      

%%%%%%%%%%%%%%%%%%%%%%%%%%%%%%%%%%%%%%%%%%%%%%%%%%
%%%%%%%%%%%%%%%%%%%%%%%%%%%%%%%%%%%%%%%%%%%%%%%%%%
%%%%%%%%%%%%%%%%%%%%%%%%%%%%%%%%%%%%%%%%%%%%%%%%%%
%%%%%%%%%%%%%%%%%%%%%%%%%%%%%%%%%%%%%%%%%%%%%%%%%%
%%%%%%%%%%%%%%%%%%%%%%%%%%%%%%%%%%%%%%%%%%%%%%%%%%
%%%%%%%%%%%%%%%%%%%%%%%%%%%%%%%%%%%%%%%%%%%%%%%%%%

\subsection{The ideals of $\mathbb{Q}(\sqrt{-3})$}\label{sec_Prelim_Ideals}
Let $K=\mathbb{Q}(\sqrt{-3})$ and let $\omega = \frac{1+ \sqrt{-3}}{2}$. The ring of integers of $K$ is $\mathcal{O}_K = \mathbb{Z}[\omega]$, which is a principal ideal domain. Therefore, all fractional ideals of $K$ are also principal, and are of the form $\frac{1}{m} (\alpha)$ where $m \in \mathbb{Z} \setminus \{0\}$ and $\alpha \in \mathcal{O}_K$.

We will be interested in ideals of $\mathcal{O}_K$ of norm $p$, a prime. Let $\mathfrak{I} =  (x+y \omega)$ be an ideal of $\mathcal{O}_K$. Then $N(\mathfrak{I}) = N(x+y \omega) = x^2+xy+y^2$. If $p = x^2+ xy +y^2 = (x-y)^2 + 3xy$, then $p$ is a square modulo 3 and so $p=3$ or $p\equiv 1 \pmod 6$. It is well known \cite{C} that a prime $p$ can be written as $p=a^2+3b^2$, $a$ and $b$ integers, uniquely, up to the sign of $a$ and $b$, if and only if $p=3$ or $p\equiv 1 \pmod 6$. If we let $x= a-b$ and $y=2b$ then $x^2+xy+y^2 =p$. So there exists an ideal $\mathfrak{I}$ with $N(\mathfrak{I})=p$ if and only if $p=3$ or $p\equiv 1 \pmod 6$.

Now we determine how many ideals of norm $p$ there are when $p=3$, and when $p\equiv 1 \pmod 6$. 
Let $S_1 = \{(a,b) \mid a+3b^2 =p \}$. We define an equivalence relation, $\sim_1$ on $S_1$ by
$$(a_1, b_1) \sim_1 (a_2,b_2) \Longleftrightarrow |a_1|=|a_2|, |b_1|=|b_2|.$$ Then the order of $S_1/\sim_1$ is one. 
Next we let $S_2$ be the set of all ideals of $\mathcal{O}_K$ which have norm $p$. We define an equivalence relation, $\sim_2$ on $S_2$ by
$$(x_1 + y_1 \omega) \sim_2 (x_2 +y_2 \omega) \Longleftrightarrow (x_2 +y_2 \omega) \in \{(x_1 +y_1\omega),(y_1 +x_1 \omega)\}.$$ Note that each equivalence class in $S_2/\sim_2$ contains at most two elements.
We have a bijective map $\phi: S_1/\sim_1 \to S_2/\sim_2$ given by $ \phi([a,b]) = [a-b + 2b \omega]$, with inverse
$$\phi^{-1}([x+y\omega])
=\begin{cases}
[x+\frac{y}{2}, \frac{y}{2}] & \text{if $x$ odd, $y$ even},\\[3pt]
[y+\frac{x}{2}, \frac{x}{2}] & \text{if $x$ even, $y$ odd},\\[3pt]
[\frac{x-y}{2}, \frac{x+y}{2}] & \text{if $x$ odd, $y$ odd}.\\
\end{cases}
$$
Therefore $S_2/\sim_2$ contains only one equivalence class. The group of units of $\mathcal{O}_K$ is cyclic of order 6 generated by $\omega$. It is easy then to check that $(x+y\omega) = (y+x \omega)$ only if $p=3$, when $x=y$. Hence, when $p=3$ there is only one ideal of norm $p$ and when $p\equiv 1 \pmod 6$ there are two ideals of norm $p$.
Later, it will be convenient for us to choose the representatives of these ideal as
$-1+2 \omega = \sqrt{-3}$ in the case that $p=3$, and,
$(a-b)+2b \omega = a+b\sqrt{-3}$ and $(a+b)-2b\omega = a-b\sqrt{-3}$ in the case $p\equiv 1 \pmod 6$, where $p=a^2+3b^2$ with $a\equiv 1 \pmod 3$ and $b>0$.

%%%%%%%%%%%%%%%%%%%%%%%%%%%%%%%%%%%%%%%%%%%%%%%%%%
%%%%%%%%%%%%%%%%%%%%%%%%%%%%%%%%%%%%%%%%%%%%%%%%%%
%%%%%%%%%%%%%%%%%%%%%%%%%%%%%%%%%%%%%%%%%%%%%%%%%%
%%%%%%%%%%%%%%%%%%%%%%%%%%%%%%%%%%%%%%%%%%%%%%%%%%
%%%%%%%%%%%%%%%%%%%%%%%%%%%%%%%%%%%%%%%%%%%%%%%%%%
%%%%%%%%%%%%%%%%%%%%%%%%%%%%%%%%%%%%%%%%%%%%%%%%%%

\subsection{The Ap{\'e}ry-like numbers}\label{sec_Prelim_Apery}

The $A(n)$ Ap{\'e}ry numbers defined in (\ref{AperyNumbers}) satisfy the three term recurrence relation
\begin{equation}\label{for_recurrence}
b\, (n+1)^2 \, u(n+1)+(an^2+an-\lambda) \, u(n)+n^2 \, u(n-1)=0,
\end{equation}
when $(a,b,\lambda)=(11,-1,-3)$ and $u(0)=1$.
Following Beukers \cite{Be2}, Zagier \cite{Z} conducted a search involving over 100 million suitable triples $(a,b,\lambda)$ and found 36 triples which yielded an integral solution to (\ref{for_recurrence}) with $u(0)=1$. Six of these solutions are classed as sporadic and all six have binomial sum representations. The sequences $C(n)$ and $D(n)$, defined in (\ref{def_Cn}) and (\ref{def_Dn}), are two of these sporadic cases. The $B(n)$ Ap{\'e}ry numbers also satisfy a three term recurrence relation and generalized searches similar to Zagier's have been conducted for this relation also \cite{AZ, Co}. In general, integer sequences which are  solutions to either of these recurrences are known as Ap{\'e}ry-like numbers. Congruence properties of many of these sequences have been studied by various authors \cite{CCS, D, MS, OS, OS2, OSS, SB, V}.

Of particular interest to us are the results in \cite{SB}. For $M\in \{2,3,4\}$, let $U_M(n)$ be the sequences defined by $U_2(n)= (-1)^{\frac{n-1}{2}} D\left(\frac{n-1}{2}\right)$, $U_3(n)= C\left(\frac{n-1}{2}\right)$, and $U_4(n)= (-1)^{\frac{n-1}{2}} A\left(\frac{n-1}{2}\right)$ when $n$ is odd, and $U_M(n)=0$ when $n$ is even. Then for primes $p$ not dividing $M$ and all positive integers $m$ and $r$, we have
\begin{equation}\label{for_U_gamma}
U_M(mp^r) \equiv
\begin{cases}
(4a^2-2p) \, U_M(mp^{r-1}) - p^2 \, U_M(mp^{r-2})  \pmod{p^r} & \textup{ if }  \left(\frac{-M}{p}\right) = 1,\\
 p^2 \, U_M(mp^{r-2}) \pmod{p^r} & \textup{ if } \left(\frac{-M}{p}\right) = -1,
\end{cases}
\end{equation}
where $p=a^2+Mb^2$, $a$ and $b$ integers, when  $\left(\frac{-M}{p}\right) = 1$.
We will see later that $4a^2-2p$ equals $\beta_3(p)$, $\gamma_3(p)$ and $\alpha_3(p)$ when $M=2,3$ and $4$ respectively.
Thus, (\ref{for_U_gamma}) gives us a congruence relation between values of $D(n)$, $C(n)$, and $A(n)$ and the Fourier coefficients of $g_3(z)$, $f_3(z)$, and $h_3(z)$ respectively, which will be central to our proofs of Theorems \ref{thm_CongCn}, \ref{thm_CongCn2} and \ref{thm_CongCn3}.

%%%%%%%%%%%%%%%%%%%%%%%%%%%%%%%%%%%%%%%%%%%%%%%%%%
%%%%%%%%%%%%%%%%%%%%%%%%%%%%%%%%%%%%%%%%%%%%%%%%%%
%%%%%%%%%%%%%%%%%%%%%%%%%%%%%%%%%%%%%%%%%%%%%%%%%%
%%%%%%%%%%%%%%%%%%%%%%%%%%%%%%%%%%%%%%%%%%%%%%%%%%
%%%%%%%%%%%%%%%%%%%%%%%%%%%%%%%%%%%%%%%%%%%%%%%%%%
%%%%%%%%%%%%%%%%%%%%%%%%%%%%%%%%%%%%%%%%%%%%%%%%%%

\section{Proofs}\label{sec_Proofs}

\begin{proof}[Proof of Theorem \ref{thm_CM}]
For $k$ in each equivalence class modulo $6$, we will define a Hecke character $\Psi_k$ and construct the required CM newform $f_k$, using the methodology outlined in Section~\ref{sec_Prelim_CM}.

For an ideal $\mathfrak{f} \in \mathcal{O}_K$, let $I(\mathfrak{f})$ denote the group of fractional ideals prime to $\mathfrak{f}$, and let $J(\mathfrak{f})$ be the subset of principal fractional ideals whose generator is multiplicatively congruent to $1$ modulo $\mathfrak{f}$, i.e., $J(\mathfrak{f})=\{ (\alpha) \in I(\mathfrak{f}) \mid \alpha \equiv^{\times} 1 \pmod {\mathfrak{f}} \}$.

Let $K=\mathbb{Q}(\sqrt{-3})$, which has discriminant $D=-3$ and whose ring of integers is $\mathcal{O}_K = \mathbb{Z}\left[\frac{1+ \sqrt{-3}}{2}\right]$, which is a principal ideal domain. Therefore, all fractional ideals of $K$ are also principal, and are of the form $\frac{1}{m} (\alpha)$ where $m \in \mathbb{Z} \setminus \{0\}$ and $\alpha \in \mathcal{O}_K$. 

We let $\omega := \frac{1+ \sqrt{-3}}{2}$ for brevity. In what follows, $p$ will always denote a prime. When $p\equiv 1 \pmod{6}$ we will reserve $a$ and $b$ exclusively to be the integers defined by $p=a^2+3b^2$ with $a\equiv 1 \pmod{3}$ and $b>0$. Note that $a$ and $b$ must have different parity. We also note the discussion in Section \ref{sec_Prelim_Ideals}, where we express the ideals of $\mathcal{O}_K$ which have norm $p\equiv 1\pmod 6$ in terms of $a$, $b$ and $\omega$.

Case 1: $k \equiv 1 \pmod 6$. Let $\mathfrak{f}= (1)$. Then $I(\mathfrak{f})=J(\mathfrak{f})$ is the set of all fractional ideals. We define the Hecke character $\Phi_k : I((1)) \to \mathbb{C}^{*}$ of weight $k$ and conductor $(1)$ by
$$
\Phi_k \left( \tfrac{1}{m} (\alpha) \right)
=
\left( \tfrac{\alpha}{m}  \right)^{k-1}.
$$
Therefore, 
$$
f_k(z) := \sum_{\mathfrak{a}} \Phi_k(\mathfrak{a}) q^{N(\mathfrak{a})} = \sum_{n=1}^{\infty} \gamma_k(n) q^n \in S_k(\Gamma_0(3),(\tfrac{-3}{\cdot}))
$$
is a CM newform and, for $p$ an odd prime, 
$$
\gamma_k(p) = 
\begin{cases}
(a-b+2b \omega)^{k-1} + (a+b-2b \omega)^{k-1} & \text{if $p \equiv 1 \imod 6$},\\
(-3)^\frac{k-1}{2} & \text{if $p =3$}\\
0 & \text{otherwise}.
\end{cases}
$$
Note here that $(-1+2 \omega)^2 = -3$.

Case 2: $k \equiv 4 \pmod 6$. Let $\mathfrak{f}= (1+\omega)$. Then
$$
I((1+\omega)) = \{ \tfrac{1}{m} (\alpha) \mid \alpha=x+y \omega, x+2y \not\equiv 0 \imod{3}, m \not\equiv 0 \imod{3}\}
$$
and
$$
J((1+\omega)) = \{ \tfrac{1}{m} (\alpha) \mid \alpha=x+y \omega, x+2y \equiv m \not\equiv 0 \imod{3}\}.
$$
We define the Hecke character $\Phi_k : I((1+\omega)) \to \mathbb{C}^{*}$ of weight $k$ and conductor $(1+\omega)$ by
$$
\Phi_k \left( \tfrac{1}{m} (x+y\omega) \right)
=
\left( \tfrac{x+y \omega}{m}  \right)^{k-1} \chi\left(\tfrac{1}{m}(x+y\omega)\right),
$$
where
\begin{equation}\label{def_chi}
\chi\left(\tfrac{1}{m}(x+y\omega)\right)=
\begin{cases}
+1, & \text{if } x+2y \equiv m \imod{3},\\
-1, & \text{if } x+2y \not\equiv m \imod{3}.
\end{cases} 
\end{equation}
Therefore, 
$$
f_k(z) := \sum_{\mathfrak{a}} \Phi_k(\mathfrak{a}) q^{N(\mathfrak{a})} = \sum_{n=1}^{\infty} \gamma_k(n) q^n \in S_k(\Gamma_0(9))
$$
is a CM newform, and for $p$ an odd prime, 
$$
\gamma_k(p) = 
\begin{cases}
(a-b+2b \omega)^{k-1} + (a+b-2b \omega)^{k-1} & \text{if $p \equiv 1 \imod 6$},\\
0 & \text{otherwise}.
\end{cases}
$$

Case 3: $k\equiv 3,5 \pmod{6}$. Let $\mathfrak{f}= (2)$. Then
$$
I((2)) = \{ \tfrac{1}{m} (\alpha) \mid \alpha=x+y\omega, \text{$x$ and $y$ not both even, $m$ odd} \}
$$
and
$$J((2)) = \{ \tfrac{1}{m} (\alpha) \mid \alpha=x+y\omega, \text{$x$ odd and $y$ even, $m$ odd} \}.$$
We define the Hecke character $\Phi_k : I((2)) \to \mathbb{C}^{*}$ of weight $k$ and conductor $(2)$ by
$$
\Phi_k \left( \tfrac{1}{m} (x+y\omega) \right)
=
\left( \tfrac{x+y \omega}{m}  \right)^{k-1} \psi((x+y \omega))^{k-1},
$$
where
\begin{equation}\label{def_psi}
\psi((x+y \omega))=
\begin{cases}
1 & \text{if $x$ odd, $y$ even},\\
\omega^2 & \text{if $x$ even, $y$ odd},\\
\omega^4 & \text{if $x$ odd, $y$ odd}.
\end{cases} 
\end{equation}
Therefore, 
$$
f_k(z) := \sum_{\mathfrak{a}} \Phi_k(\mathfrak{a}) q^{N(\mathfrak{a})} = \sum_{n=1}^{\infty} \gamma_k(n) q^n \in S_k(\Gamma_0(12),(\tfrac{-3}{\cdot}))
$$
is a CM newform and, for $p$ an odd prime, 
$$
\gamma_k(p) = 
\begin{cases}
(a-b+2b \omega)^{k-1} + (a+b-2b \omega)^{k-1} & \text{if $p \equiv 1 \imod 6$},\\
(-3)^\frac{k-1}{2} & \text{if $p =3$}\\
0 & \text{otherwise}.
\end{cases}
$$

Case 4: $k \equiv 0,2 \pmod 6$. Let $\mathfrak{f}= (2+2\omega)$. Then
$$
I((2+2\omega)) = \{ \tfrac{1}{m} (\alpha) \mid \alpha=x+y \omega, \text{$x$ and $y$ not both even}, x+2y \not\equiv 0 \imod{3}, \text{$m$ odd}, m \not\equiv 0 \imod{3}\}
$$
and
$$
J((2+2\omega)) =\{ \tfrac{1}{m} (\alpha) \mid \alpha=x+y \omega, \text{$x$ odd and $y$ even}, \text{$m$ odd}, x+2y \equiv m \not\equiv 0 \imod{3}\}.
$$
We define the Hecke character $\Phi_k : I((2+2\omega)) \to \mathbb{C}^{*}$ of weight $k$ and conductor $(2+2\omega)$ by
$$
\Phi_k \left( \tfrac{1}{m} (x+y\omega) \right)
=
\left( \tfrac{x+y \omega}{m}  \right)^{k-1} \chi\left(\tfrac{1}{m}(x+y\omega)\right) \psi((x+y \omega))^{k-1},
$$
where $\chi(\cdot)$ and $\psi(\cdot)$ are defined as in (\ref{def_chi}) and (\ref{def_psi}).
Therefore, 
$$
f_k(z) := \sum_{\mathfrak{a}} \Phi_k(\mathfrak{a}) q^{N(\mathfrak{a})} = \sum_{n=1}^{\infty} \gamma_k(n) q^n \in S_k(\Gamma_0(36))
$$
is a CM newform, and for $p$ an odd prime, 
$$
\gamma_k(p) = 
\begin{cases}
(a-b+2b \omega)^{k-1} + (a+b-2b \omega)^{k-1} & \text{if $p \equiv 1 \imod 6$},\\
0 & \text{otherwise}.
\end{cases}
$$
So now we have proved all but the last line of the theorem. By the multiplicative properties of the Fourier coefficients of newforms it suffices to show that $\gamma_k(p)$ is an integer for all primes $p$. This is obvious when $p \not\equiv 1 \pmod{6}$. We now examine the case when $p\equiv 1 \mod 6$. We have shown that for all $k$,
\begin{align*}
\gamma_k(p) 
&=(a+b \sqrt{-3})^{k-1} + (a-b \sqrt{-3})^{k-1}\\
&=\sum_{t=0}^{k-1} \binom{k-1}{t} a^{k-1-t} b^t  (\sqrt{-3})^t + \sum_{t=0}^{k-1} \binom{k-1}{t} a^{k-1-t}  (-1)^ t b^t  (\sqrt{-3})^t\\
&= \sum_{t=0}^{k-1} \binom{k-1}{t} a^{k-1-t} b^t   (\sqrt{-3})^t \left[ 1+  (-1)^t \right] \\
&= \sum_{\substack{t=0\\t \; even}}^{k-1} \binom{k-1}{t} a^{k-1-t} b^t (-3)^{\frac{t}{2}} ,
\end{align*}
which is an integer
\end{proof}

%%%%%%%%%%%%%%%%%%%%%%%%%%%%%%%%%%%%%%%%%%%%%%%%%%
%%%%%%%%%%%%%%%%%%%%%%%%%%%%%%%%%%%%%%%%%%%%%%%%%%

\begin{proof}[Proof of Corollary \ref{cor_Coeffs}]
For $p=3$ the result follows directly from the formula for $\gamma_k(3)$ in Theorem \ref{thm_CM}. When $p\equiv 5 \pmod 6$, both sides equal zero, as $\gamma_k(p)=0$ for all $k \geq 2$. 
We now assume $p \equiv 1 \pmod{6}$. Let
$\gamma^{+}(p) =a+b \sqrt{-3}$
and
$\gamma^{-}(p) =a-b \sqrt{-3}.$
From Theorem \ref{thm_CM} we get that
\begin{align*}
\gamma_k(p)^r 
&= \left( \gamma^{+}(p)^{k-1} + \gamma^{-}(p)^{k-1} \right)^r\\
&= \sum_{t=0}^{r} \binom{r}{t}  \gamma^{+}(p)^{t(k-1)} \, \gamma^{-}(p)^{(r-t)(k-1)} \\
&= \sum_{t=0}^{\lfloor \frac{r-1}{2} \rfloor} \binom{r}{t} \left( \gamma^{+}(p)^{r(k-1)} \, \gamma^{-}(p)^{(r-t)(k-1)} + \gamma^{+}(p)^{(r-t)(k-1)}  \, \gamma^{-}(p)^{r(k-1)} \right)\\
&\qquad \qquad \qquad \qquad \qquad \qquad \qquad \qquad +
\begin{cases}
\displaystyle \binom{r}{\frac{r}{2}} \left( \gamma^{+}(p) \, \gamma^{-}(p) \right) ^{\frac{r}{2}(k-1)}, & \text{if $r$ is even},\\
0, & \text{if $r$ is odd}.
\end{cases}
\end{align*}
To complete the proof, we write the summand as
\begin{equation*}
\binom{r}{t} \left( \gamma^{+}(p) \, \gamma^{-}(p) \right)^{t(k-1)} 
\left( \gamma^{-}(p)^{(r-2t)(k-1)} + \gamma^{+}(p)^{(r-2t)(k-1)} \right)
\end{equation*}
and note that 
\begin{equation*}
  \gamma^{-}(p)^{(r-2t)(k-1)} + \gamma^{+}(p)^{(r-2t)(k-1)} = \gamma_{(r-2t)(k-1)+1}(p)
\end{equation*}
and that $\gamma^{+}(p) \, \gamma^{-}(p)=p$.
\end{proof}

%%%%%%%%%%%%%%%%%%%%%%%%%%%%%%%%%%%%%%%%%%%%%%%%%%
%%%%%%%%%%%%%%%%%%%%%%%%%%%%%%%%%%%%%%%%%%%%%%%%%%

\begin{proof}[Proof of Corollary \ref{cor_CongCoeffs}]
This follows from reducing the result in Corollary \ref{cor_Coeffs} modulo $p^{k-1}$.
\end{proof}

%%%%%%%%%%%%%%%%%%%%%%%%%%%%%%%%%%%%%%%%%%%%%%%%%%
%%%%%%%%%%%%%%%%%%%%%%%%%%%%%%%%%%%%%%%%%%%%%%%%%%

\begin{proof}[Proof of Theorem \ref{thm_CongCn}]

Let $p$ be prime. When $p \equiv 1 \pmod{6}$, let $p =a^2 +3b^2 $, with $a\equiv 1 \imod 3$ and $b>0$ integers. We first show that $4a^2-2p = \gamma_3(p)$, the $p$-th Fourier coefficient of the weight three CM form $f_3(z)$ constructed in Theorem \ref{thm_CM}. From Theorem \ref{thm_CM} we get that 
\begin{align*}
\gamma_3(p) &=(a+b \sqrt{-3})^{2} +(a-b \sqrt{-3})^{2}\\
&=2a^2-6b^2\\
&=4a^2-2p.
\end{align*}
Then, taking $m=r=1$ in (\ref{for_U_gamma}), and using the facts that $U_3(1)=1$ and $\gamma_3(p)= 0$ when $p\equiv 5 \pmod{6}$, we see that for all primes $p>3$,
\begin{equation}\label{for_U3_gamma2}
U_3\left(p\right) \equiv \gamma_3(p) \pmod{p}.
\end{equation}
Substituting (\ref{for_U3_gamma2}) back into (\ref{for_U_gamma}) with $m=1$ and reducing modulo $p$ we have 
\begin{equation*}
U_3\left(p^r\right) \equiv U_3\left(p\right) U_3\left(p^{r-1} \right) \pmod{p}
\end{equation*}
for all positive integers $r$. Using a simple inductive argument we then get that for all primes $p>3$ and integers $r\geq1$,
\begin{equation}\label{for_UrU3}
U_3\left(p^r\right) \equiv U_3\left(p\right)^r \pmod{p}.
\end{equation}
Accounting for (\ref{for_U3_gamma2}) yields
\begin{equation*}
U_3\left(p^r\right) \equiv \gamma_3(p)^r \pmod{p},
\end{equation*}
which completes the proof as $U\left(p^r\right) = C\left(\frac{p^r-1}{2}\right)$ and $\gamma_3(p)^r \equiv \gamma_{2r+1}(p) \pmod{p}$ by Corollary \ref{cor_CongCoeffs}. It is worth noting also that  (\ref{for_UrU3}) gives us the nice relation
$$C\left(\frac{p^r-1}{2}\right) \equiv C\left(\frac{p-1}{2}\right)^r \pmod{p}.$$

\end{proof}

%%%%%%%%%%%%%%%%%%%%%%%%%%%%%%%%%%%%%%%%%%%%%%%%%%
%%%%%%%%%%%%%%%%%%%%%%%%%%%%%%%%%%%%%%%%%%%%%%%%%%

\begin{proof}[Proof of Theorem \ref{thm_CM2}]
Let $K=\mathbb{Q}(\sqrt{-2})$, which has discriminant $D=-8$ and whose ring of integers is $\mathcal{O}_K = \mathbb{Z}\left[\sqrt{-2}\right]$, which is a principal ideal domain. Therefore, all fractional ideals of $K$ are also principal, and are of the form $\frac{1}{m} (\alpha)$ where $m \in \mathbb{Z} \setminus \{0\}$ and $\alpha \in \mathcal{O}_K$. 

It is well known \cite{C} that a prime $p$ can be written as $p=c^2+2d^2$, $c$ and $d$ integers, uniquely, up to the sign of $c$ and $d$, if and only if $p=2$ or $p\equiv 1 \text{ or } 3 \pmod 8$. Let $\mathfrak{I} =  (c+d \sqrt{-2})$ be an ideal of $\mathcal{O}_K$. Then $N(\mathfrak{I}) = c^2+2d^2$. The units of $\mathcal{O}_K$ are $\pm 1$,and therefore, for a prime $p$ there are two ideals of norm $p$ when $p\equiv 1,3 \pmod{8}$ and one when $p=2$. 

Let $\mathfrak{f}= (1)$. Then $I(\mathfrak{f})=J(\mathfrak{f})$ is the set of all fractional ideals. We define the Hecke character $\Phi_k : I((1)) \to \mathbb{C}^{*}$ of weight $k$ and conductor $(1)$ by
$$
\Phi_k \left( \tfrac{1}{m} (\alpha) \right)
=
\left( \tfrac{\alpha}{m}  \right)^{k-1}.
$$
Therefore, 
$$
g_k(z) := \sum_{\mathfrak{a}} \Phi_k(\mathfrak{a}) q^{N(\mathfrak{a})} = \sum_{n=1}^{\infty} \beta_k(n) q^n \in S_k(\Gamma_0(8),(\tfrac{-8}{\cdot}))
$$
is a CM newform and, for $p$ a prime, 
$$
\beta_k(p) = 
\begin{cases}
(c+d \sqrt{-2})^{k-1} + (c-d \sqrt{-2})^{k-1} & \text{if $p \equiv 1,3 \imod 8$},\\
(-2)^\frac{k-1}{2} & \text{if $p =2$}\\
0 & \text{otherwise}.
\end{cases}
$$
We now note that when $p \equiv 1,3 \pmod 8$,
\begin{align*}
\beta_k(p) 
&=(c+d \sqrt{-2})^{k-1} + (c-d \sqrt{-2})^{k-1}\\
&= \sum_{t=0}^{k-1} \binom{k-1}{t} c^{k-1-t} d^t (\sqrt{-2})^t \left[1 + (-1)^t) \right] \\
&= \sum_{\substack{t=0\\t \; even}}^{k-1} \binom{k-1}{t} c^{k-1-t} d^t (-2)^{\frac{t}{2}} 
\end{align*}
which is an integer.
\end{proof}

%%%%%%%%%%%%%%%%%%%%%%%%%%%%%%%%%%%%%%%%%%%%%%%%%%
%%%%%%%%%%%%%%%%%%%%%%%%%%%%%%%%%%%%%%%%%%%%%%%%%%

\begin{proof}[Proof of Corollary \ref{cor_Coeffs2}]
Proceed along the same lines as in the proof of Corollary \ref{cor_Coeffs} but with $\beta^{+}(p) =(c+d \sqrt{-2})$
and
$\beta^{-}(p) = (c-d \sqrt{-2}).$
\end{proof}

%%%%%%%%%%%%%%%%%%%%%%%%%%%%%%%%%%%%%%%%%%%%%%%%%%
%%%%%%%%%%%%%%%%%%%%%%%%%%%%%%%%%%%%%%%%%%%%%%%%%%

\begin{proof}[Proof of Corollary \ref{cor_CongCoeffs2}]
This follows from reducing the result in Corollary \ref{cor_Coeffs2} modulo $p^{k-1}$.
\end{proof}

%%%%%%%%%%%%%%%%%%%%%%%%%%%%%%%%%%%%%%%%%%%%%%%%%%
%%%%%%%%%%%%%%%%%%%%%%%%%%%%%%%%%%%%%%%%%%%%%%%%%%

\begin{proof}[Proof of Theorem \ref{thm_CongCn2}]

Let $p$ be prime. When $p \equiv 1,3 \pmod{8}$, let $p =c^2 +3d^2 $, with $c$ and $d$ integers. 
From Theorem \ref{thm_CM2} we get that 
\begin{align*}
\beta_3(p) &= (c+d \sqrt{-2}))^{2} + (c-d \sqrt{-2})^{2}
=2a^2-4b^2
=4a^2-2p.
\end{align*}
Then, taking $m=r=1$ in (\ref{for_U_gamma}), and using the facts that $U_2(1)=1$ and $\beta_3(p)= 0$ when $p\equiv 5,7 \pmod{8}$, we see that for all primes $p>2$,
\begin{equation}\label{for_U2_gamma2}
U_2\left(p\right) \equiv \beta_3(p) \pmod{p}.
\end{equation}
Substituting (\ref{for_U2_gamma2}) back into (\ref{for_U_gamma}) with $m=1$ and reducing modulo $p$ we have 
\begin{equation*}
U_2\left(p^r\right) \equiv U_2\left(p\right) U_2\left(p^{r-1} \right) \pmod{p}
\end{equation*}
for all positive integers $r$. Using a simple inductive argument we then get that for all primes $p>2$ and integers $r\geq1$,
\begin{equation}\label{for_UrU2}
U_2\left(p^r\right) \equiv U_2\left(p\right)^r \pmod{p}.
\end{equation}
Accounting for (\ref{for_U2_gamma2}) yields
\begin{equation*}
U_2\left(p^r\right) \equiv \beta_3(p)^r \pmod{p},
\end{equation*}
which completes the proof as $U_2\left(p^r\right) = (-1)^{\frac{p^r-1}{2}} D\left(\frac{p^r-1}{2}\right)$ and $\beta_3(p)^r \equiv \beta_{2r+1}(p) \pmod{p}$ by Corollary \ref{cor_CongCoeffs2}. It is worth noting also that (\ref{for_UrU2}) gives us the nice relation
$$D\left(\frac{p^r-1}{2}\right) \equiv D\left(\frac{p-1}{2}\right)^r \pmod{p},$$
as  $(-1)^{\frac{p^r-1}{2}}= (-1)^{\frac{r(p-1)}{2}}$.
\end{proof}

\begin{proof}[Proof of Theorem \ref{thm_CongCn3}]
Following the description of $h_k(z)$ in Section \ref{sec_Intro}, we note that $\alpha_3(p)=4a^2-2p$ when $p \equiv 1 \pmod4$. Note also that, for $p\geq5$, Ahlgren's result (\ref{superapery1}) gives us
$$
(-1)^{\frac{(p-1)}{2}}  U_4(p) = A\Bigl(\frac{p-1}{2}\Bigr) \equiv \alpha_3(p) \pmod{p^2}.
$$
So, considering (\ref{for_U_gamma}) modulo $p^2$ with $m=1$ and $r\geq2$ we get that
\begin{equation*}
U_4\left(p^r\right) \equiv U_4\left(p\right) U_4\left(p^{r-1}\right) \pmod{p^2},
\end{equation*}
which holds for all $p\geq 5$, as $\alpha_3(p)=0$ when $p \equiv 3 \pmod4$, and $(-1)^{\frac{(p-1)}{2}} =1$ when $p \equiv 1 \pmod 4$.
Using a simple inductive argument we then get that for all primes $p\geq5$ and integers $r\geq2$,
\begin{equation}\label{for_UrU4}
U_2\left(p^r\right) \equiv U_2\left(p\right)^r \pmod{p^2},
\end{equation}
which implies
\begin{equation*}
(-1)^{\frac{(p^r-1)}{2}} A\Bigl(\frac{p^r-1}{2}\Bigr) \equiv \alpha_3(p)^r \pmod{p^2}.
\end{equation*}
This completes the proof as $\alpha_3(p)^r \equiv \alpha_{2r+1}(p) \pmod{p^2}$ from (\ref{for_alphap2}), and $(-1)^{\frac{(p^r-1)}{2}}=1$ when $p \equiv 1 \pmod4$. 
\end{proof}

%%%%%%%%%%%%%%%%%%%%%%%%%%%%%%%%%%%%%%%%%%%%%%%%%%
%%%%%%%%%%%%%%%%%%%%%%%%%%%%%%%%%%%%%%%%%%%%%%%%%%
%%%%%%%%%%%%%%%%%%%%%%%%%%%%%%%%%%%%%%%%%%%%%%%%%%
%%%%%%%%%%%%%%%%%%%%%%%%%%%%%%%%%%%%%%%%%%%%%%%%%%

\section*{Acknowledgements}
The second author is supported by a grant from the Simons Foundation (\#353329, Dermot McCarthy).

%%%%%%%%%%%%%%%%%%%%%%%%%%%%%%%%%%%%%%%%%%%%%%%%%%
%%%%%%%%%%%%%%%%%%%%%%%%%%%%%%%%%%%%%%%%%%%%%%%%%%
%%%%%%%%%%%%%%%%%%%%%%%%%%%%%%%%%%%%%%%%%%%%%%%%%%
%%%%%%%%%%%%%%%%%%%%%%%%%%%%%%%%%%%%%%%%%%%%%%%%%%

\end{document}